\documentclass{amsart}
\usepackage{fullpage}
\usepackage[foot]{amsaddr}
\usepackage{tikz}
\usepackage{hyperref}
\usepackage{amsmath}
\usepackage{pgfmath}
\usepackage{amsthm}
\usepackage{amsfonts}
\usepackage{amssymb}
\usepackage{graphicx}
\usepackage{float}
\usepackage{ytableau}
\usepackage[all]{xy}
\usepackage[normalem]{ulem}

\newcommand{\stkout}[1]{\ifmmode\text{\sout{\ensuremath{#1}}}\else\sout{#1}\fi}

\newtheorem{theorem}{Theorem}[subsection]
\newtheorem*{theorem*}{Theorem}
\newcounter{tmp}

\newtheorem{corollary}[theorem]{Corollary}
\newtheorem{definition}{Definition}[subsection]

\newtheorem{lemma}[theorem]{Lemma}

\newtheorem{remark}[theorem]{Remark}

\newcommand{\CC}{\mathbb{C}}
\newcommand{\RR}{\mathbb{R}}

\newcommand{\ZZ}{\mathbb{Z}}

\newcommand{\ind}{\mathrm{ind}}

\newcommand{\frakg}{\mathfrak{g}}

\newcommand{\frakp}{\mathfrak{p}}
\newcommand{\frakl}{\mathfrak{l}}
\newcommand{\frakq}{\mathfrak{q}}
\newcommand{\fraku}{\mathfrak{u}}
\newcommand{\frakn}{\mathfrak{n}}
\newcommand{\frakm}{\mathfrak{m}}

\newcommand{\diag}{\mathrm{diag}}

\title{Littlewood-Richardson Coefficient, Springer Fibers and the Annihilator Varieties of Induced Representations}
\author{Zhuohui Zhang\\Weizmann Institute of Science}
\email{zhuohui.zhang@weizmann.ac.il} 
\begin{document}
\begin{abstract}
    For $G=GL(n,\mathbb{C})$ and a parabolic subgroup $P=LN$ with a two-block Levi subgroup $L=GL(n_1)\times GL(n_2)$, the space $G\cdot (\mathcal{\mathcal{O}}+\frakn)$, where $\mathcal{O}$ is a nilpotent orbit of $\frakl$, is a union of nilpotent orbits of $\frakg$. In the first part of our main theorem, we use the geometric Sakate equivalence to prove that $\mathcal{O'}\subset G\cdot (\mathcal{\mathcal{O}}+\frakn)$ if and only if some Littlewood-Richardson coefficients do not vanish. The second part of our main theorem describes the geometry of the space $\mathcal{O}\cap\frakp$, which is an important space to study for the Whittaker supports and annihilator varieties of representations of $G$.
\end{abstract}
\maketitle

\section{Introduction}\label{introduction}
\indent A nilpotent orbit of a reductive group $G$ over $\CC$ is an orbit of the adjoint or the coadjoint action of $G$ on the nilpotent cone $\mathcal{N}\subset\frakg$ or $\frakg^*$, respectively. The set of nilpotent orbits $\mathcal{O}\subset\mathcal{N}$ forms a poset with a partial order $\mathcal{O}_1 < \mathcal{O}_2$ defined by the closure order $\mathcal{O}_1 \subset \overline{\mathcal{O}_2}$. For classical groups, there is a bijection between the nilpotent orbits and certain integer partitions. In particular, for $GL(n)$ and $SL(n)$, each partition $\alpha$ of the integer $n$ corresponds to a nilpotent orbit $\mathcal{O}_\alpha$. The counterpart of the closure order $\mathcal{O}_\alpha\subset\overline{\mathcal{O}_\beta}$ on partitions is given by
\[
    \alpha <\beta \text{ if and only if } \sum_{i=1}^k\alpha_i\leq \sum_{i=1}^k\beta_i\text{ for all }k.
\]
\indent The notion of \emph{induced orbits} provides a connection between the nilpotent orbits of a subgroup and the orbits of the ambient group. Following \cite[Chapter 7]{collingwoodbook}, for any nilpotent orbit $\mathcal{O}$ of a Levi subgroup $L\subset G$, there is an unique nilpotent orbit of $G$ which intersects $\mathcal{O} + \frakn$ in a Zariski-open subset. This orbit is called the \emph{induced orbit} $\ind_\frakl^\frakg(\mathcal{O})$ in $G$. On the other hand, we can define the \emph{Bala-Carter} inclusion $\mathrm{inc_\frakl^\frakg(\mathcal{O})}$ of a nilpotent orbit $\mathcal{O}$ of $\frakl$ as the $G$-saturation $G\cdot \mathcal{O}$ in $\frakg$. These two orbits are the maximal and minimal elements in the poset which consists of all the $G$-orbits contained in $G\cdot (\mathcal{O}+\frakn)$. The first part of the main result of this paper describes this poset of the nilpotent orbits contained in $G\cdot (\mathcal{O}+\frakn)$.\\
\indent We can also vary the set $G\cdot (\mathcal{O}+\frakn)$ by changing the choices of parabolic subgroups $P$ and the unipotent radicals $N\subset P$. For each pair of orbits $\mathcal{O}_\gamma$ of $\frakl$ and $\mathcal{O}_\alpha$ of $\frakg$, we will also describe an algebraic variety parametrizing all the parabolic subgroups $P=LN$ such that
\[
    \mathcal{O}_\alpha\subset G\cdot \left(\mathcal{O}_\gamma+\frakn\right).
\]
The collection of all such conjugate parabolic subgroups forms a closed subvariety of the partial flag manifold $G/P$, and some information of its topology is encoded in the Littlewood-Richardson coefficients. This variety is closely related to the geometry of $\mathcal{O}_\gamma\cap\frakp$, and we will provide this result as the second part of our main theorem. In Section \ref{appl}, we will also discuss the application of this geometric result to the wavefront set of subrepresentations of an induced representation, following a result in \cite{GourevitchSayag}.\\
\indent Now we state the main theorem of this paper:
\begingroup
    \setcounter{tmp}{\value{theorem}}
    \setcounter{theorem}{0}
    \renewcommand\thetheorem{\Alph{theorem}}
    \begin{theorem}
        \label{maintheorem}
        For $G = GL(n,\mathbb{C})$ and a parabolic subgroup $P=LN$ with a standard Levi subgroup $L = GL(n_1,\mathbb{C})\times GL(n_2,\mathbb{C}) \subset G$ embedded block-diagonally, let $\alpha$ be a partition of $n_1$ and $\beta$ be a partition of $n_2$, consider the nilpotent orbit $\mathcal{O}_{\alpha,\beta}=\mathcal{O}_{\alpha}\times \mathcal{O}_{\beta}$ of the subgroup $L$ where $\mathcal{O}_{\alpha},\mathcal{O}_\beta$ are the nilpotent orbits corresponding to the partitions $\alpha,\beta$ in $GL(n_1)$ and $GL(n_2)$, repsectively, then:
        \begin{enumerate}
            \item A $G$-orbit $\mathcal{O}_\gamma$ is contained in $G\cdot (\mathcal{O}_{\alpha,\beta}+\frakn)$ if and only if the Littlewood-Richardson coefficient $c_{\alpha,\beta}^\gamma\neq 0$.
            \item The Littlewood-Richardson coefficient is the number of irreducible components of the subvariety
            \[
                \mathcal{P}_{\alpha,\beta}^\gamma = \left\{x\in \mathcal{O_\gamma}\cap \frakp\mid p_\frakl(x)\in\mathcal{O}_\alpha\times \mathcal{O}_\beta\right\}\subset \mathcal{O}_\gamma\cap\frakp
            \]
            where the map $p_\frakl$ is the projection $p_\frakl: \frakp\mapsto \frakp/\frakn\cong \frakl$ onto the Levi subgroup.
        \end{enumerate}
    \end{theorem}   
\endgroup
The first part of this theorem is a direct consequence of the theory of Hall polynomials which was discussed for abelian $p$-groups in \cite[Appendix Chapter II, Theorem AZ.1]{macdonaldbook}, and can be proven as a corollary of the Hall's theorem in \cite[Theorem 2.2]{klein1969hall}. The proof of Hall's theorem is combinatorial and uses Schubert calculus. However, in order to find an interpretion of the Littlewood-Richardson coefficients in our context, we would like to give a geometric proof to the first part of the theorem as a corollary of the geometric Satake isomorphism. The geometric constructions used in the proof of the first part will shed light on the proof of the second part of Theorem \ref{maintheorem} in Section \ref{nilpchapter}.\\
\indent It is worth mentioning that since we can induce the orbits by stages, the first part of Theorem \ref{maintheorem} can be generalized to arbitrary standard parabolic subgroups:
\begingroup
    \setcounter{tmp}{\value{theorem}}
    \setcounter{theorem}{1}
    \renewcommand\thetheorem{\Alph{theorem}}
\begin{corollary}
    For any partition $\underline{n}=(n_1,\ldots,n_r)$ of $n$ and partitions $\alpha_i$ of the integers $n_i$, we take a standard parabolic subgroup $P_{\underline{n}}=L_{\underline{n}}N_{\underline{n}}$ with a Levi subgroup isomorphic to $GL(n_1)\times\ldots\times GL(n_r)$. A nilpotent orbit $\mathcal{O}_\gamma$ of $GL(n)$ is contained in $G\cdot (\prod\mathcal{O}_{\alpha_i}+\frakn_{\underline{n}})$ if and only if the number
    \[
        N^\gamma_{\alpha_1,\ldots,\alpha_r}=\sum_{\beta_2,\ldots,\beta_{r-1}}c_{\alpha_1,\alpha_2}^{\beta_2}c_{\beta_2,\alpha_3}^{\beta_3}c_{\beta_3,\alpha_4}^{\beta_4}\ldots c_{\beta_{r-1},\alpha_r}^{\gamma}
    \]
    is nonzero.
\end{corollary}
In terms of Schur functions, the number $N_{\alpha_1,\ldots,\alpha_r}^\gamma$ is the coefficient of $\mathbf{s}_\gamma$ in the product of the Schur polynomials $\mathbf{s}_{n_1}\ldots \mathbf{s}_{n_r}$.
\endgroup

\subsection*{Acknowledgements}
I thank Dmitry Gourevitch for his support, enlightening discussions and valuable comments about this project. This research is supported by the Israel Science Foundation (ISF) 249/17 and ERC Starting Grants 637912.
\section{Preliminaries and Notation}
 In this section, we specify the notations which are used throughout this paper. Let $G$ be a semisimple algebraic group over $\CC$, we denote by
\begin{itemize}
	\item $T\subset B\subset G$ a choice of maximum torus $T$ and a Borel subgroup $B$,
	\item $X_*(T)$ the cocharacter lattice, and $X^*(T)$ the character lattice of $T$,
	\item $\Delta_G(T)$ the set of roots, $\Delta_{G,+}(T)$ the set of positive roots, $\Delta_{G,s}(T)$ the set of simple roots, and $\Delta^G(T)$ the set of coroots.
\end{itemize}
We can identify the character lattice as the integer lattice $\ZZ^n$ with $n$ being the rank of $G$. Each partition $\lambda=(\lambda_1,\ldots,\lambda_n)$ corresponds to the character
\[
	(t_1,\ldots,t_n)\mapsto t_1^{\lambda_1}\ldots t_n^{\lambda_n}
\]
on $T$. In particular, under this correspondence, the half sum of positive roots 
\[
    \rho = \frac{1}{2}\sum_{\alpha\in \Delta_{G,+}}\alpha
\]
corresponds to the partition $(n-1,n-2,\ldots,1,0)$. Fixing a maximal torus $T$, the quadruple $(X^*,\Delta^G,X_*,\Delta_G)$ is called the \emph{root datum} of $G$. Its dual root datum $(X_*,\Delta_G, X^*,\Delta^G)$ is obtained by switching the character lattice/roots and the cocharacter lattice/coroots. The algebraic group corresponding to the dual root datum is called the \emph{Langlands dual group} and is denoted by $\check{G}$.

\subsection{Young Diagrams}\label{LittlewoodRichardsonSection}
In this section we will summarize the basic concepts related to Young diagrams and define the Littlewood-Richardson coefficients. A partition $\alpha=(\alpha^1,\ldots,\alpha^r)$ of an integer $n$ corresponds to a Young diagram (whose $i$-th row has length $\alpha^i$) of the shape $\alpha$. In a \emph{Young tableau} of the shape $\alpha$, integers, starting with 1, are filled into the boxes of the Young diagram $\alpha$. We denote the set of all Young tableaux of the shape $\alpha$ by $\mathcal{T}_\alpha$. A Young tableau is  called \emph{standard} if the entries in each box are strictly increasing from the left to right along each row, and increasing from top to bottom along each column. If the entries are weakly increasing along the rows, the tableau is called a \emph{semistandard} tableau. It is worth noting that each standard Young tableau $T$ of the shape $\alpha$ corresponds to a sequence of Young diagrams
\[
    0= \alpha_0\subset\alpha_1\subset\ldots\subset\alpha_r = \alpha
\]
with each $\alpha_i$ formed by the boxes in $T$ with contents $\leq i$. For example, the Young tableau
\begin{center}
    $T=$
    \ytableausetup{smalltableaux}
    \begin{ytableau}
        1 & 1 & 2\\
        1 & 3\\
        2\\
    \end{ytableau}
\end{center}
corresponds to the following sequence of Young diagrams:
\begin{center}
    $\emptyset\subset$ 
    \ydiagram{2,1} $\subset$
    \ydiagram{3,1,1} $\subset$
    \ydiagram{3,2,1}.
\end{center}

\indent Similar to the usual Young diagrams discussed above, for each pair of Young diagrams $(\alpha,\gamma)$ such that $\alpha\subset\gamma$, a \emph{skew Young diagram} is the complement of $\alpha$ in $\gamma$, with the two diagrams aligned along their top and left edges. For example, for $\alpha=(2,1)$ and $\gamma = (4,3,2)$, the skew Young diagram $\gamma/\alpha$ looks like
\begin{center}
    \ydiagram{2+2,1+2,0+2}.
\end{center}
\indent We denote the set of all semistandard skew tableaux of the shape $\gamma/\alpha$ by $\mathcal{T}_{\gamma/\alpha}$. Similar to the correspondence between Young tableaux and sequences of Young diagrams described above, any skew tableau $T\in\mathcal{T}_{\gamma/\alpha}$ corresponds to a sequence of Young diagrams from $\alpha$ to $\gamma$, whose $i$-th stage $\alpha_i$ consists of the boxes in $T$ such that the entries in each of the selected boxes are $\leq i$:
\[
    \alpha = \alpha_0 \subset \alpha_1 \subset \alpha_2\subset\ldots\subset \alpha_r =\gamma.
\]
The entries of such a skew tableau also gives us a new partition 
\[
    \beta = (|\alpha_1/\alpha_0|,\ldots,|\alpha_r/\alpha_{r-1}|).
\] 
For example, the skew tableau
\begin{center}
    \ytableausetup{smalltableaux}
    \ytableaushort{\none\none\none11,\none\none\none2,\none12,12}
\end{center}
yields the partition $\beta = (4,3)$. We say a tableau $T$, or a corresponding sequence of Young diagrams from $\alpha$ to $\gamma$, is of type $(\alpha,\beta;\gamma)$ if this new partition is $\beta$.

\subsection{The Littlewood-Richardson Rule}
For each pairs of Young diagrams $(\alpha,\gamma)$ such that $\alpha\subset \gamma$, a semistandard skew Young tableau is \emph{Littlewood-Richardson} if its \emph{reverse lattice words}, i.e. the string of content of each box, read from top to bottom and from right to left like in Hebrew, satisfies the \emph{Yamanouchi condition}, i.e. for each content $i$, the number of $i$'s in the length-$k$ prefix of its reverse lattice word is great or equal to the number of $i+1$'s in the length-$k$ prefix. For example, the reverse lattice word of the tableau
\begin{center}
    \ytableausetup{smalltableaux}
    \begin{ytableau}
        \none & \none & 1&1\\
        \none & 1 & 2\\
        2 & 3
    \end{ytableau}
\end{center}
is $112132$, which satisfies the Yamanouchi condition. For any triple of Young diagrams $(\alpha,\beta;\gamma)$ such that $\alpha,\beta\subset\gamma$, the \emph{Littlewood-Richardson coefficient} $c_{\alpha,\beta}^\gamma$ is defined as the number of Littlewood-Richardson tableaux of type $(\alpha,\beta;\gamma)$.\\
\indent An algebraic interpretation of the Littlewood-Richardson coefficient can be given by the Hall-Littlewood polynomials $G^\gamma_{\alpha,\beta}(q)$. Considering a pair $(\mathcal{O},\mathfrak{p})$ of a discrete valuation ring $\mathcal{O}$ and its maximal ideal $\mathfrak{p}$ with a finite residue field of size $q$, for any triple of partitions $(\alpha,\beta;\gamma)$, denoting by $M_\alpha$ the direct sum of cyclic modules $M_\alpha = \bigoplus_{i}\mathcal{O}/\frakp^{n_i}$ for any partition $\alpha = (n_1,\ldots,n_r)$, $G^\gamma_{\alpha,\beta}(q)$ is defined as the number of submodules $N\subset M_{\gamma}$ such that $N\cong M_{\alpha}$ and $M_{\gamma}/N\cong M_{\beta}$. In fact, $G^\gamma_{\alpha,\beta}(q)$ is a polynomial in $q$ with degree $\langle\rho,\alpha+\beta-\gamma\rangle$ called the \emph{Hall-Littlewood polynomial}, and its top-degree coefficient is equal to the Littlewood-Richardson coefficient $c^\gamma_{\alpha,\beta}$. A detailed proof of this fact can be found in the appendix of \cite[Chapter 2]{macdonaldbook}.\\
\indent However, it's important to have a geometric model of the abelian group extensions so that we can obtain more information related to nilpotent orbits. For this purpose, in the following sections, we will introduce the theory of affine Grassmannian and a geometric interpretation of the Littlewood-Richardson coefficient.

\section{Lattices and Affine Grassmannian}
\label{affinegrassmannian}
In this paper, we will discuss four interpretations for the Littlewood-Richardson coefficient: the first picture is related to the extensions of torsion $\mathbb{C}[t]$-modules and the top-degree coefficient of Hall-Littlewood polynomials, which we have already introduced in the previous section. The second picture relates the Littlewood-Richardson coefficients to extensions of sublattices in $\mathcal{O}^n$. This approach is closely related to the first picture, and requires the understanding of the geometry of affine Grassmannian. The third picture applies the Geometric Satake Isomorphism to draw connection between the Littlewood-Richardson coefficients and the decomposition multiplicities of tensor products of finite dimensional representations of $GL_n$. The fourth picture shows that the Littlewood-Richardson coefficients is equal to the number of irreducible components of certain subvarieties of the Springer fibers in a generalized flag manifold. The whole Section \ref{affinegrassmannian} is devoted to the discussion of the second and the third picture. The fourth picture will be introduced in Section \ref{nilpchapter}.\\
\indent Throughout this paper, we will fix the ring $\mathcal{O}=\mathbb{C}\left[\!\left[t\right]\!\right]$ with a maximal ideal $\frakp = t\mathcal{O}$. We will introduce a geometric model for the extensions of the finite dimensional $\mathcal{O}$-module
\begin{equation}\label{defm}
    M_{\alpha} = \CC[t]/(t^{\alpha_1})\oplus\ldots \CC[t]/(t^{\alpha_n}).
\end{equation}
This module is the counterpart over $\mathcal{O}$ of the direct sum of cyclic modules in the definition of the Hall-Littlewood polynomial for a partition $\alpha = \left(\alpha_1,\ldots,\alpha_n\right)$. We can add zeros to the end of $\alpha$ without causing any changes to $M_\alpha$. The following lemma relates the Littlewood-Richardson coefficient to the extensions of module of type $M_\alpha$: 
\begin{lemma}\label{mainlemma}
	For three partitions $\alpha,\beta,\gamma$ and three modules $M_\alpha,M_\beta,M_\gamma$ as defined above, there exists an exact sequence
	\[
	0\rightarrow M_{\alpha}  \rightarrow M_{\gamma } \rightarrow M_{\beta}  \rightarrow 0
	\]
	if and only if the partitions $\alpha,\beta,\gamma$ make the Littlewood-Richardson coefficient $c_{\alpha,\beta}^{\gamma}$ non-vanishing. 
\end{lemma}
A purely combinatorial proof of this lemma can be found in \cite[Theorem 4.1-4.2]{klein1968}. The first part of Theorem \ref{maintheorem} is a direct consequence of this lemma. In the later sections of this paper, we will seek a geometric proof of this lemma using affine grassmannians, which will prepare for us the tools we will use to prove the second part of Theorem \ref{maintheorem}.

\subsection{Preliminaries on Affine Grassmannian}\label{prelimaffine}


The \emph{affine Grassmannian} is a geometric object parametrizing the full-rank $\mathcal{O}$-sublattices in $\mathcal{O}^n$, which will be used in the geometric proof of Lemma \ref{mainlemma}.
\begin{definition}
	Setting $\mathcal{O}=\mathbb{C}\left[\!\left[t\right]\!\right]$ and $K=\mathbb{C}\left(\!\left(t\right)\!\right)$, which are the rings of formal power series and formal Laurent series, respectively, let $G$ be a semisimple algebraic group, the \emph{affine Grassmannian} of $G$ is defined as the quotient
	\[
		\mathrm{Gr}_{G}=G\left(K\right)/G\left(\mathcal{O}\right).
	\]
\end{definition}
The affine Grassmannian is an \emph{ind-scheme}, which means that it is the direct limit of finite-dimensional closed subschemes. There exists a stratification on the affine Grassmannian based on the \emph{Cartan decomposition} (See \cite[Section 2]{MirkovicVilonen} and \cite[Proposition 1.3.2]{BaumannRiche}) of $G$, in which the cells are parametrized by the dominant coweights
\begin{equation}
	G(K) =\coprod_{\lambda\in X_{*}\left(T\right)^{+}} G\left(\mathcal{O}\right)\cdot L_{\lambda}\cdot G\left(\mathcal{O}\right).
\end{equation}
In the case $G = GL_n$, $L_\lambda$ is the diagonal matrix $L_\lambda = \mathrm{diag}\left(t^{\lambda_1},\ldots,t^{\lambda_n}\right)$
where $\lambda=(\lambda_1,\ldots,\lambda_n)$ is a partition of length $n$. We define an \emph{affine Schubert cell} as the quotient
\[
    \mathrm{Gr}_G^\lambda = G\left(\mathcal{O}\right)\cdot L_{\lambda}\cdot G\left(\mathcal{O}\right)/G\left(\mathcal{O}\right).
\]
The closure $\overline{\mathrm{Gr}_G^\lambda}$ of this cell is called an \emph{affine Schubert variety}. By \cite[Proposition 1.3.2]{BaumannRiche}, it is the disjoint union of affine Schubert cells $\mathrm{Gr}_G^\mu$ with $\mu\leq \lambda$:
\[
	\overline{\mathrm{Gr}_G^\lambda} = \coprod_{\substack{\mu\in X_{*}\left(T\right)^{+}\\ \mu\leq \lambda}} \mathrm{Gr}_G^\mu.
\]
The order $\mu\leq \lambda$ on the partitions is given by the closure order described in the beginning of Section \ref{introduction}.

\subsection{Affine Schubert Varieties and Lattices in $\mathcal{O}^n$}\label{latticesubsection}
We will use the affine Schubert varieties to parametrize full-rank $\mathcal{O}$-sublattices in $\mathcal{O}^n$. The correspondence is established in
 the following lemma:
\begin{lemma}\label{Lattice1}
	For any dominant coweight $\lambda = (\lambda_1,\ldots,\lambda_n)\in X_*(T)^+$, denoting by $\Lambda_x$ the lattice generated by the column vectors of $x\in \mathrm{Gr}_G^\lambda$, then there is a bijection between the elements in $x\in \mathrm{Gr}_G^\lambda$ and full-rank $\mathcal{O}$-sublattices $\Lambda_x \subset \mathcal{O}^n$ such that
	\[
		\mathcal{O}^n/\Lambda_x \cong \bigoplus_i \CC[t]/(t^{\lambda_i}).
	\]
\end{lemma}
\begin{proof}
	Since each point $x\in\mathrm{Gr}_G^\lambda$ can be represented by an element of the form
	\[
		x = \alpha L_{\lambda},\; \alpha\in G(\mathcal{O})
	\]
    in ${G(\mathcal{O})}L_\lambda{G(\mathcal{O})}$, the column vectors of $L_\lambda =\diag\left(t^{\lambda_1},\ldots,t^{\lambda_n}\right)$ generates a sublattice
	\[
		\alpha^{-1}\Lambda_x = t^{\lambda_{1}}\mathcal{O}\oplus\ldots\oplus t^{\lambda_{n}}\mathcal{O}
	\]
	of $\alpha^{-1}\mathcal{O}^n$. The quotient $\alpha^{-1}\left(\mathcal{O}^n/\Lambda_x\right) \cong \mathcal{O}^n/\Lambda_x$ is a module which satisfies the requirement of the lemma. Conversely, if there is a lattice $\Lambda = \bigoplus_i\mathcal{O}f_i$ satisfying the requirement of the lemma, the isomorphism in the statement of the lemma will give rise to an isomorphism between lattices
	$
		p:\bigoplus_j\mathcal{O}t^{\lambda_j}\rightarrow \Lambda =\bigoplus_i\mathcal{O}f_i 
	$
	such that 
	$
		f_j = \sum_i r_{ji} t^{\lambda_j}
	$
	where the matrix $\alpha = (r_{ji})$ is an invertible matrix in $G(\mathcal{O})$. Setting $x = \alpha L_\alpha$, then we can see that $\Lambda = \Lambda_x$ as defined in the statement of the lemma.
\end{proof}

In order for us to describe the set of extensions of lattices as mentioned in the beginning of Section \ref{prelimaffine}, we need to introduce the \emph{twisted product} $\mathrm{Gr}_G\widetilde{\times}\mathrm{Gr}_G$ of two copies of the affine Grassmannians.
Letting $G(\mathcal{O})$ act on $G(K)\times \mathrm{Gr}_G$ by $k\cdot(a,b)=(ak,k^{-1}b)$, the space $\mathrm{Gr}_G\widetilde{\times}\mathrm{Gr}_G$ is defined as the orbits in $G(K)\times \mathrm{Gr}_G$ of such an action by $G(\mathcal{O})$, with the quotient map by the action given by $q$:
\begin{equation}\label{twistedprod}
    \xymatrix{
     &G(K)\times \mathrm{Gr}_G\ar[ld]_p\ar[rd]^q&\\
     \mathrm{Gr}_G\times \mathrm{Gr}_G && \mathrm{Gr}_G\widetilde{\times} \mathrm{Gr}_G
    }
\end{equation}
The morphism $p$ is the quotient by $G(\mathcal{O})$ acting on the first entry. There is a 
multiplication map $m:\mathrm{Gr}_G\tilde{\times}\mathrm{Gr}_G\rightarrow \mathrm{Gr}_G$ defined by
\[
    m: (g,x)\mapsto gx.
\]
The reason why this map $m$ is called the multiplication map will be explained in Lemma \ref{Lattice2}.\\

\indent The twisted product $\widetilde{\times}$ can be defined similarly for the Schubert cells and Schubert varieties by simply taking the quotient space by the $G(\mathcal{O})$-action on the direct product of the corresponding Schubert cells and Schubert varieties. When restricted to the twisted product of the Schubert varieties, the multiplication map $m$ maps $\overline{\mathrm{Gr}_G^{\alpha}\widetilde{\times }\mathrm{Gr}_G^{\beta}}$ surjectively onto $\overline{\mathrm{Gr}_G^{\alpha+\beta}}$. \\
\indent We can also take the twisted product of multiple affine Grassmannians and Schubert varieties. Denote by $\mathrm{Gr}_{G}^{\lambda^{1},\ldots,\lambda^{n}}$ the multiple twisted product
\[
    \mathrm{Gr}_{G}^{\lambda^{1},\ldots,\lambda^{n}} = \mathrm{Gr}_{G}^{\lambda^1}\widetilde{\times}\ldots \widetilde{\times}\mathrm{Gr}_{G}^{\lambda^n}.
\]
The variety $\overline{\mathrm{Gr}_{G}^{\lambda^{1},\ldots,\lambda^{n}}}$ is stratified by the twisted product of Schubert cells corresponding to smaller partitions:
\[
    \overline{\mathrm{Gr}_{G}^{\lambda^{1},\ldots,\lambda^{n}}} = \bigcup_{\mu^i\leq \lambda^i} \mathrm{Gr}_G^{\mu^1,\ldots,\mu^n}.
\]
To count the dimension of the twisted product of these Schubert cells, we need the following lemma from \cite{XinwenZhu} and \cite{BaumannRiche}:
\begin{lemma}
    The dimension of the twisted product of cells $\mathrm{Gr}_{G}^{\lambda^{1},\ldots,\lambda^{n}}$ is equal to $\langle 2\rho,\sum_{i=1}^n\lambda^i\rangle$. The dimension of the fiber $m^{-1}(x)\cap \mathrm{Gr}_{G}^{\lambda^{1},\ldots,\lambda^{n}}$ of $m$ over any generic point $x\in \mathrm{Gr}_G^{\lambda}\subset\mathrm{Gr}_G$ is less than or equal to $\langle \rho,\sum_{i=1}^n\lambda^i-\lambda\rangle$.
\end{lemma}
\begin{proof}
    Since the stablizer of the action of $G(\mathcal{O})$ at $L^\lambda$ is $G(\mathcal{O})\cap L_\lambda G(\mathcal{O}) (L_\lambda)^{-1}$, thus the tangent space $\mathrm{Gr}_G^\lambda$ at $L_{\lambda}$ is
    \[
        \frakg(\mathcal{O})/\frakg(\mathcal{O})\cap \mathrm{Ad}_{L_\lambda} \frakg(\mathcal{O})\cong \bigoplus_{\langle\alpha,\lambda\rangle\geq 0}\frakg_\alpha(\mathcal{O})/t^{\langle\alpha,\lambda\rangle}\frakg(\mathcal{O}).
    \]
    The dimension of this tangent space above is equal to $\langle 2\rho,\lambda\rangle$. The dimension of the twisted product follows from a similar stabilizer argument. The inequality for the dimension of the fiber $m^{-1}(x)\cap \mathrm{Gr}_{G}^{\lambda^{1},\ldots,\lambda^{n}}$ follows from \cite[Lemma 4.4]{MirkovicVilonen} and also \cite[Lemma 6.4]{BaumannRiche}.
\end{proof}
\begin{remark}\label{dimensionrmk}
    For any partition $\lambda$, the number $\langle\rho,\lambda\rangle$ is closely related to $\frac{1}{2}\left\|\lambda^t\right\|^2 = \frac{1}{2}\sum \left(\lambda^t_i\right)^2$. Since we know
    \[
        \lambda_i-\lambda_{i+1} = \#\left|\left\{j\mid \lambda^t_j=i \right\}\right|,
    \]
    we can express $\frac{1}{2}\left\|\lambda^t\right\|^2$ in terms of $\lambda$:
    \[
        \frac{1}{2}\left\|\lambda^t\right\|^2 = \sum_{i=1}i^2(\lambda_i-\lambda_{i+1}).
    \]
    Through simple calculation, we see that
    \[
        \frac{1}{2}\left\|\lambda^t\right\|^2+\langle\rho,\lambda\rangle = \left(n-\frac{1}{2}\right)\sum_{k=1}^n \lambda_k.
    \]
    Therefore, we can express the dimension of the fiber $m^{-1}(x)\cap \mathrm{Gr}_{G}^{\lambda^{1},\ldots,\lambda^{n}}$ differently in terms of the conjugate partitions:
    \begin{equation}
        \dim \left(m^{-1}(x)\cap \mathrm{Gr}_{G}^{\lambda^{1},\ldots,\lambda^{n}}\right) = \frac{1}{2}\left(\left\|\lambda^t\right\|^2-\sum_i\left\|\lambda_i^t\right\|^2\right).\label{dimensionnew}
    \end{equation}
\end{remark}
We can use the variety $\mathrm{Gr}_G^{\alpha}\widetilde{\times }\mathrm{Gr}_G^{\beta}$ to parametrize the extensions of lattices (as promised in the second picture mentioned in the beginning of Section \ref{affinegrassmannian}) in $\mathcal{O}^n$. The correspondence is presented in the following lemma (recalling the definition of the $\mathcal{O}$-module $M_\alpha$ in (\ref{defm})): 
\begin{lemma}\label{Lattice2}
The variety $\mathrm{Gr}_G^{\alpha}\tilde{\times }\mathrm{Gr}_G^{\beta}$ parametrizes the collection of rank-$n$ sublattices $\Lambda_{\gamma}\subset \Lambda_{\alpha}\subset \mathcal{O}^{n}$, with $\gamma$ a partition satisfying $|\gamma| = |\alpha| + |\beta|$, such that
\begin{enumerate}
    \item $\mathcal{O}^n/\Lambda_\gamma\cong M_\gamma$, and $\mathcal{O}^n/\Lambda_\alpha\cong M_\alpha$,
    \item $\Lambda_{\alpha}/\Lambda_\gamma\cong M_{\beta}$.
\end{enumerate}
The multiplication map $m: \mathrm{Gr}_G^{\alpha}\tilde{\times }\mathrm{Gr}_G^{\beta}\rightarrow \overline{\mathrm{Gr}_G^{\alpha+\beta}}$ sends any pair of such lattices $(\Lambda_{\gamma},\Lambda_{\alpha})$ to $\Lambda_{\gamma}$.
\end{lemma}
\begin{proof}
    The scheme $\mathrm{Gr}_G^{\alpha}\widetilde{\times }\mathrm{Gr}_G^{\beta}$ is a $G(\mathcal{O})$-quotient of the product $G(\mathcal{O})L_{\alpha}G(\mathcal{O})\times G(\mathcal{O})L_{\beta}$, which parametrizes the pairs of elements $(l_1,l_2)$ with $l_2\in G(\mathcal{O})L_{\beta}$ and $l_1\in G(\mathcal{O})L_{\alpha}G(\mathcal{O})$. Each orbit of the action by $G(\mathcal{O})$ has a unique representative of the form $(x L_{\alpha},y L_{\beta})$, and lies in the same $G(\mathcal{O})$-orbit as the representative $(xL_{\alpha}y,L_{\beta})$. In the lattice $x\mathcal{O}^n$, one can consider the lattice generated by the column vectors of the matrix $xL_{\alpha}y L_{\beta}$. We would like to prove that this lattice is the correct $\Lambda_{\gamma}$ in the statement of this lemma.\\
    \indent The lattice $\Lambda_\alpha$ is generated by the column vectors of $xL_\alpha$. In the lattice $x^{-1}\Lambda_\alpha$, we can write the diagonal matrix $L_\alpha = \left(t^{\alpha_1},\ldots,t^{\alpha_n}\right)$ in the form
    \[
        L_{\alpha} = (e_1,\ldots,e_n),
    \]
    where each $e_i$ is a column vector $e_i = (0,\ldots,t^{\alpha_i},\ldots,0)^t$. Rewriting the matrix $y L_{\beta}$ as $y L_{\beta} = (g_{i,j})$, the column vectors of the matrix $L_{\alpha}y L_{\beta}$ are thus given by \[v_j = (t^{\alpha_i}g_{i,j}) = \sum_{i}g_{i,j}e_{i}.\]
    The vectors $v_j$ generate a sublattice $x^{-1}\Lambda_\gamma$ of $x^{-1}\Lambda_\alpha = \sum_i\mathcal{O}e_i$ which satisfies $\Lambda_{\alpha}/\Lambda_\gamma\cong M_{\beta}$.\\
    \indent Conversely, for each pair of rank-$n$ lattices $(\Lambda_\gamma,\Lambda_\alpha)$ satisfying the conditions of the lemma, since by Lemma \ref{Lattice1} there exist an element $x\in \mathrm{Gr}_G^\alpha$ such that $\Lambda_\alpha = \Lambda_x$. By Bruhat decomposition, $x$ can be represented by a matrix $x'L_{\alpha}$ with $x'\in G(\mathcal{O})$ with $\Lambda_x$ generated by the column vectors of $x'L_{\alpha}$. Also under the isomorphism $\mathcal{O}^n\cong\Lambda_x$ sending each lattice element represented by a column vector $z$ to $x'L_\alpha z$, there also exists an element $y\in \mathrm{Gr}_G^\beta$ represented by $y'L_{\beta}$ such that $x'L_\alpha y'L_\beta$ generates $\Lambda_\gamma$.
\end{proof}

\subsection{Proof of the Main Lemma}

For a generic $x\in \mathrm{Gr}_G^\gamma$, the properties of the fiber $m^{-1}(x)\cap \mathrm{Gr}_G^{\alpha,\beta}$ of the multiplication map is studied in \cite{haines2003structure} and \cite{haines2006equidimensionality} as corollaries of the Geometric Satake Isomorphism. By the representation theory of finite dimensional representations of $GL(n)$, the multiplicity of the finite representation $V^\gamma$ with highest weight $\gamma$ has multiplicity $c_{\alpha,\beta}^\gamma$ in $V^\alpha\otimes V^\beta$. By \cite[Theorem 2.2]{haines2006equidimensionality}, 
\begin{align}\label{satakeformula}
    c_{\alpha,\beta}^\gamma = 
    \text{\# irreducible components of }m^{-1}(x)\cap \mathrm{Gr}_G^{\alpha,\beta}\text{ with dimension }\langle\rho,\alpha+\beta-\gamma\rangle.
\end{align}
for a generic $x\in\mathrm{Gr}_G^\gamma$.

Lemma \ref{mainlemma} follows as a corollary of this result:
\begin{proof}[Proof of Lemma \ref{mainlemma}]
By Lemma \ref{Lattice1}, any element $x\in \mathrm{Gr}_G^\gamma$ represents a lattice $\Lambda_{x}$ such that the quotient $\mathcal{O}^n/\Lambda_x$ has Jordan type $\gamma$. If the multiplication map
\[
    m:\mathrm{Gr}_G^\alpha\widetilde{\times}\mathrm{Gr}_G^\beta\rightarrow\overline{\mathrm{Gr}_G^{\alpha+\beta}}
\]
sends a pair $y = (\Lambda_\gamma,\Lambda_\alpha)$ to $\Lambda_x$, by Lemma \ref{Lattice2} we must have $\Lambda_\gamma = \Lambda_x$. 
Since by (\ref{satakeformula}), on each Schubert cycle $\mathrm{Gr}_G^\gamma$, the group $G(K)$ acts transitively, the preimage of the Schubert cycle $m^{-1}(\mathrm{Gr}_G^\gamma)$ intersects $\mathrm{Gr}^{\alpha,\beta}_G$ if and only if the fiber $m^{-1}(x)$ intersects $\mathrm{Gr}^{\alpha,\beta}_G$. 
Since all the dominant coweights of $GL_n$ are sums of coweights of the form $\epsilon^i = \left(1^{i},0^{n-i}\right)$, which are all \emph{minuscule}, i.e. $\langle\epsilon_i,\alpha\rangle\in\{-1,0,1\}$ for any root $\alpha$. By \cite[Theorem 1.3]{haines2006equidimensionality}, the irreducible components of $m^{-1}(x)\cap \mathrm{Gr}^{\alpha,\beta}_G$, if not empty, are equidimensional and of dimension $\langle\rho,\alpha+\beta-\gamma\rangle$. Since the Littlewood-Richardson coefficient $c_{\alpha,\beta}^\gamma$ is equal to the number of irreducible components of $m^{-1}(x)\cap \mathrm{Gr}^{\alpha,\beta}_G$, we have now proven that $m^{-1}(x)\cap \mathrm{Gr}^{\alpha,\beta}_G\neq\emptyset$ if and only if $c_{\alpha,\beta}^\gamma\neq 0$.
\end{proof}

\section{Interpretations of the Littlewood-Richardson Coefficient}\label{nilpchapter}
In this section, we consider the Young diagrams with shapes $\alpha,\beta\subset \gamma$ such that $c_{\alpha,\beta}^{\gamma}\neq 0$. We will fix our notations to denote the length of $\gamma$ by $m$, the length of $\alpha$ by $n$. The goal of this section is to prove the second part of the Theorem \ref{maintheorem}, which follows from the geometric interpretation of the Littlewood-Richardson coefficients studied by Springer in \cite{SpringerConstruction} and Marc van Leeuwen in \cite{VanLeeuwen}. This geometric interpretation of $c_{\alpha,\beta}^\gamma$ is our fourth picture mentioned in the beginning of Section \ref{prelimaffine}.
\subsection{The Satake Fiber}
The fiber $\mathcal{S}^\gamma_{\alpha,\beta}(x) = m^{-1}(x)\cap \mathrm{Gr}^{\alpha,\beta}_G$ of the multiplication map $m$ above the point $x\in\mathrm{Gr}_\gamma$ is sometimes referred to as the \emph{Satake fiber}. The following lemma identifies the Satake fiber as a subvariety of the partial flag variety $G/P_{[n,m-n]}$:
\begin{lemma}\label{fiberisom}
    Fixing a vector space $W$ of dimension $|\alpha|+|\beta|$ and a nilpotent matrix $x$ acting on $W$ with Jordan type $\gamma$, the Satake fiber $\mathcal{S}^\gamma_{\alpha,\beta}(x)$ parametrizes all subspaces $V$ of dimension $|\alpha|$ inside $W$ with the following properties:
    \begin{enumerate}
        \item $V$ is stabilized by $x$,
        \item $x\mid_V$ has Jordan type $\alpha$,
        \item $x$ acts on $W/V$ with Jordan type $\beta$.
    \end{enumerate}
\end{lemma}
\begin{proof}
    The Satake fiber $m^{-1}(L_\gamma)\cap \mathrm{Gr}^{\alpha,\beta}_G$ parametrizes the pairs of lattices $(\Lambda_\gamma,\Lambda_\alpha)$ satisfying the conditions in Lemma \ref{Lattice2}. The quotients $M_\alpha,M_\beta,M_\gamma$ of $\mathcal{O}^n$ by the three lattices $\Lambda_\alpha,\Lambda_\beta$ and $\Lambda_\gamma$, respectively, fit into the exact sequence
    \[
        0\rightarrow M_\alpha\rightarrow M_\gamma\rightarrow M_\beta\rightarrow 0.
    \]
    The multiplication by the variable $t$ of the ring $\mathcal{O}[\![t]\!]$ is a nilpotent linear operator with Jordan type $\alpha,\beta,\gamma$ on the three spaces, respectively. Conversely, if we have a pair of spaces $V\subset W$ satisfying the conditions of this lemma, $V$ and $W$ can be viewed as the $\mathcal{O}$-modules $M_\alpha,M_\gamma$, respectively.
\end{proof}
\subsection{The Springer-type Fibers}
\[
    T^*\mathcal{B} = \left\{(x,gB)\mid x\in \mathrm{Ad}(g)\mathfrak{b}\right\}\subset \mathcal{N}\times \mathcal{B}.
\]
For each nilpotent orbit $\mathcal{O}_\gamma\subset\mathcal{N}\subset\mathfrak{g}$, the projection $p$ from the cotangent bundle $T^*\mathcal{B}$ to $\mathcal{B}$ is a resolution of singularities called the \emph{Springer resolution} for the nilpotent cone $\mathcal{N}$:
\[
    \xymatrix{ & T^*\mathcal{B}\subset \mathcal{N}\times\mathcal{B}\ar[dl]_{\pi}\ar[dr]^{p}\\
\mathcal{N} &  & \mathcal{B}
}
\]
where $\pi,p$ are the projections onto the first and the second factor, respectively. The \emph{Springer fiber} $\mathcal{B}_\gamma(x)$ above the point $x\in\mathcal{O}_\gamma$ of $p$ is the collection of Borel subalgebras $\mathfrak{b}\subset\frakg$ containing the nilpotent element $x$
\[
    \mathcal{B}_\gamma(x) = \left\{(x,gB)\mid x\in \mathrm{Ad}(g)\mathfrak{b}\right\}.
\]
It is proven by Spaltenstein in \cite{Spaltenstein} and Steinberg in \cite{Steinberg1} and \cite{Steinberg2} that the irreducible components of the Springer fiber $\mathcal{B}_\gamma(x)$ are equidimensional and are parametrized by the \emph{standard Young tableaux} of shape $\gamma$. The number of standard Young tableaux can be calculated from the hook-length formula.\\
\indent For $G=GL_m$, the Springer fiber $\mathcal{B}_\gamma(x)$ can be interpretd as the collection of $x$-stable full flags
\[
    \{0\} = V_0\subset V_1\subset V_2\subset\ldots\subset V_m = V
\]
in a vector space $V_m$ of dimension $m$, on which $x$ acts as a nilpotent linear transformation. Since the Springer fibers are isomorphic for different choices of $x$ in the same orbit $\mathcal{O}_\mu$, we can denote the Springer fiber by $\mathcal{B}_\gamma$ if there is no ambiguity caused by different choices of $x$.\\

\indent Apart from the case of the full flag variety $G/B$, we can also consider the Springer-type partial resolution from the cotangent bundle of a partial flag variety $G/P$ corresponding to the parabolic subgroup $P=LN$. Similar to the Springer resolution of the nilpotent cone by the cotangent bundles of a full flag variety, given a parabolic subgroup $P$, we can define a \emph{partial resolution} $\widetilde{\mathcal{N}_P}$ of the nilpotent cone $\mathcal{N}$
\[
    \widetilde{\mathcal{N}_P} = \left\{(x,gP)\mid x\in \mathcal{N}\cap \mathrm{Ad}(g)\frakp\right\}\subset \mathcal{N}\times G/P\xrightarrow{\pi_P}\mathcal{N}
\]
which factorizes the Springer resolution $\pi$. The \emph{generalized Springer fiber} $\pi_P^{-1}(x)$ above any element $x$ is sometimes referred to as the \emph{Steinberg variety} (cf. \cite[Chapter 3]{borho101partial}, some references also call it the \emph{Spaltenstein variety}). When $P=B$, the partial resolution $\pi_P$ becomes the usual Springer resolution, and the fiber $\pi^{-1}_P(x)$ becomes the usual Springer fiber above $x$. Equivalently, if we consider the parabolic subgroup $P$ as the stabilizer of a partial flag
\[
    V_{k_1}\subset V_{k_2} \subset \ldots\subset V_{k_r}=V_m
\]
where each $V_{k_i}$ is a subspace of dimension $k_i$. The fiber $\pi_P^{-1}(x)$ can be identified as the collection of the $x$-stable flags which are stabilized by a conjugate of $P$.\\

\indent We are particularly interested in the standard parabolic subgroups $P_{[n,1^{m-n}]}$ with a Levi subgroup $L_{[n,1^{m-n}]}\cong GL(n)\times GL(1)^{m-n}$, and $P_{[n,m-n]}$ with Levi subgroup $L_{[n,m-n]}\cong GL(n)\times GL(m-n)$. For the parabolic subgroup $P_{[n,1^{m-n}]}$, the generalized Springer fiber above $x\in\mathcal{O}_\gamma$, denoted by $\mathcal{P}_{[n,1^{m-n}]}(x)$, is the set of $x$-stable partial flags $V_{n} \subset \ldots \subset V_{m-1}\subset V_{m}$
in an $m$-dimensional vector space $V_m$ with the lowest piece $V_n$ having dimension $n$, and any two consecutive spaces satisfying $\mathrm{dim}(V_{i+1}/V_{i})=1$. The fiber $\mathcal{P}_{[n,1^{m-n}]}(x)$ can be identified as the set
\[
    \mathcal{P}_{[n,1^{m-n}]}(x) = \{(x,gP_{[n,1^{m-n}]})\mid \mathrm{Ad}(g)^{-1}x\in \mathfrak{p}_{[n,1^{m-n}]}\}.
\]
Similarly, for the parabolic subgroup $P_{[n,m-n]}$, the generalized Springer fiber above $x\in \mathcal{O}_\gamma$, denoted similarly by $\mathcal{P}_{[n,m-n]}(x)$, is a subvariety of $G/P_{[n,m-n]}$ given by
\[
    \mathcal{P}_{[n,m-n]}(x) = \{(x,gP_{n,m-n})\mid \mathrm{Ad}(g)^{-1}x\in \mathfrak{p}_{[n,m-n]}\}.
\]
\subsection{Springer Fibers, Jordan Forms, and the Proof of Part 2 of Theorem \ref{maintheorem}}

\indent For any nilpotent element $x\in\mathcal{O}_\gamma$, there is a subvariety $\mathcal{G}_{n}^m(x)$ of the Grassmannian variety $G(n,m)$ that parame\-trizes the $x$-stable subspaces of $V_{m}$ of dimension $n$. By the argument in the Section 4 of \cite{VanLeeuwen} and Lemma \ref{fiberisom}, this subvariety $\mathcal{G}_{n}^m(x)$ can be decomposed into the union of Satake fibers $\mathcal{S}^\gamma_{\alpha,\beta}(x)$ which parametrizes all $V_n\subset V_m$ of dimension $n$ such that the Jordan type of $x\mid _{V_{n}}$ on $V_{n}$ is $\alpha$, with the Jordan type of $x'$, the nilpotent linear operator acting on the quotient space $V_m/V_{n}$ induced from $x$, equal to $\beta$:
\[
    \mathcal{G}_{n}^m(x) = \bigcup_{\substack{|\alpha|=n\\|\gamma|=m\\\alpha,\beta\subset\gamma}}\mathcal{S}_{\alpha,\beta}^\gamma(x).
\]
\indent Moreover, the action of the nilpotent element $x$ with Jordan type $\beta$ on the space $V_m/V_n$ corresponds to an $x$-stable filtration
\[
    V_n\subset V_{n+1}\subset\ldots\subset V_m
\]
such that the action of $x$ on each quotient $V_{n+i}/V_{n}$ has Jordan form $\alpha_i$. By the correspondence between flags and parabolic subgroups, there is a bijection between any Satake fiber $\mathcal{S}_{\alpha,\beta}^\gamma(x)$ and the set
\[
    \left\{gP \mid \mathrm{Ad}(g)^{-1}x\in \frakp_{[n,m-n]},\text{ and }p_\frakl(\mathrm{Ad}(g)^{-1}x)\in\mathcal{O}_{\alpha,\beta}\right\}
\]
where $p_\frakl:\mathfrak{p_{[n,m-n]}}\rightarrow\mathfrak{p_{[n,m-n]}}/\fraku_{[n,m-n]}\cong\frakl_{[n,m-n]}$ is the projection map. Letting $(\alpha,\beta)$ go over all the possible pairs of Young diagrams satisfying $|\alpha|+|\beta| = |\gamma|$, the union of the spaces $\mathcal{S}_{\alpha,\beta}^\gamma(x)$ is exactly the generalized Springer fiber $\mathcal{P}_{[n,m-n]}(x)$ over a point $x\in\mathcal{O}_\gamma$ of the partial flag variety $G/P_{[n,m-n]}$. \\

\indent Now we prove the second part of the Theorem \ref{mainlemma}:
\begin{proof}
If $x\in\mathcal{O}_\gamma$ and $gP_{[n,m-n]}\in \mathcal{S}_{\alpha,\beta}^\gamma(x)$ such that $\mathrm{Ad}(g)^{-1}x\in\frakp_{[n,m-n]}\cap \mathcal{O}_{\gamma}$. The space $\mathcal{S}_{\alpha,\beta}^\gamma(x)$ is a closed subvariety of $G/P_{[n,m-n]}$ that parametrizes the parabolic subalgebras $\mathfrak{q}=\frakm+\mathfrak{u}$ conjugate to $\frakp_{[n,m-n]}$ such that
\[
    x\in \mathfrak{q}\text{ and } p_\mathfrak{m}\left(x\right)\in \mathcal{O}_{\alpha,\beta}.
\]
In \cite{SpringerConstruction}, \cite{VanLeeuwen} and the proof of the Lemma \ref{mainlemma}, it is stated that the dimensions of the irreducible components of  the Satake fiber $S^\gamma_{\alpha,\beta}(x)$ are equal to $\langle\rho,\alpha+\beta-\gamma\rangle$, and the number of such components is equal to the Littlewood-Richardson coefficient $c^\gamma_{\alpha,\beta}$.\\
\indent Consider the following partial Springer resolution $\pi_P$:
\[
    \pi_{P_{[n,m-n]}}:\widetilde{\mathcal{N}}_P\subset \mathcal{N}\times G/P_{[n,m-n]}\longrightarrow \mathcal{N},
\]
there exists a $G$-equivariant isomorphism from $G\times_{P_{[n,m-n]}} \frakp_{[n,m-n]}$ to $\widetilde{\mathcal{N}}_P$ given by the map $(g,y)\leftrightsquigarrow (\mathrm{Ad}(g)y,[gP_{[n,m-n]}])$. We denote the projection map from $G\times \frakp_{[n,m-n]}$ to $G\times_{P_{[n,m-n]}} \frakp_{[n,m-n]}$ by $\varpi$. The Satake fiber $\mathcal{S}_{\alpha,\beta}^\gamma(x)$ over a point $x\in\mathcal{O}_\gamma$ is a subvariety of $G/P_{[n,m-n]}$, and has a preimage $\varpi^{-1}\left(\pi_2^{-1}(\mathcal{S}_{\alpha,\beta}^\gamma(x))\right)$ in $G\times \frakp_{n,m-n}$.  This preimage is the set of all pairs of the form $(g,y)$ such that $\mathrm{Ad}(g)^{-1}x = y$ with $g\in \mathcal{S}_{\alpha,\beta}^\gamma(x)$. The projection of this preimage to $\mathcal{N}$ along $\pi_{P_{[n,m-n]}}$ is denoted by $\mathcal{P}_{\alpha,\beta}^\gamma$, which is a subset of $\mathcal{O}_\gamma\cap \frakp_{[n,m-n]}$. The projection from the preimage $\varpi^{-1}\left(\pi_2^{-1}(\mathcal{S}_{\alpha,\beta}^\gamma(x))\right)$ along $\pi_{P_{[n,m-n]}}$ to $\mathcal{O}_\gamma \cap \mathfrak{p}_{[n,m-n]}$ is a quotient by the action of the centralizer $C_G(x)$.\\
\indent Since the projection map from $\varpi^{-1}\left(\pi_2^{-1}(\mathcal{S}_{\alpha,\beta}^\gamma(x))\right)$ to $\mathcal{S}_{\alpha,\beta}^\gamma(x)$ is locally trivial with fiber above each point isomorphic to $P_{[n,m-n]}$, by \cite[\href{https://stacks.math.columbia.edu/tag/037A}{Tag 037A}]{stacks-project}, the number of irreducible components of the preimage $\varpi^{-1}\left(\pi_2^{-1}(\mathcal{S}_{\alpha,\beta}^\gamma(x))\right)$ is still $c_{\alpha,\beta}^\gamma$
. By (\ref{dimensionnew}) stated in Remark \ref{dimensionrmk}, the dimension of any irreducible component of $\varpi^{-1}\left(\pi_2^{-1}(\mathcal{S}_{\alpha,\beta}^\gamma(x))\right)$ is equal to
\[
    \dim \varpi^{-1}\left(\pi_2^{-1}(\mathcal{S}_{\alpha,\beta}^\gamma(x))\right)=\frac{1}{2}\left(\sum_i(\gamma^t)_i^2-\sum_i(\alpha^t)_i^2-\sum_i(\beta^t)_i^2\right)+\left(\sum\alpha_i^t\right)^2+\left(\sum\beta_i^t\right)^2 + \left(\sum\alpha_i^t\right)\left(\sum\beta_i^t\right).
\]
In the $GL(n)$ case, the centralizer $C_G(x)$ has only one irreducible component with dimension $\sum_i(\gamma^t)_i^2$. 
Thus by \cite[\href{https://stacks.math.columbia.edu/tag/037A}{Tag 037A}]{stacks-project} again, the number of irreducible components of $\mathcal{P}_{\alpha,\beta}^\gamma$ is also $c_{\alpha,\beta}^\gamma$. Since we can interpret the set $\mathcal{P}_{\alpha,\beta}^\gamma$ as the elements in $\mathcal{O}_\gamma\cap\mathfrak{p_{[n,m-n]}}$ whose projection along $p_\frakl$ lies in $\mathcal{O}_{\alpha,\beta}$, we have completed the proof of the second part of the Theorem \ref{maintheorem}.
\end{proof}
\section{Applications: Wave Front Sets and Associated Varieties}\label{appl}
In this section, let $G$ be a semisimple algebraic group over an archimedean field $F$. We would like to introduce invariants to describe of the ``growth'' of irreducible representations $\pi$ of the group $G$. We denote the Lie algebra of $G$ by $\mathfrak{g}_0$, and its complexified Lie algebra by $\mathfrak{g}$.
\subsection{Associated Varieties and Annihilator Varieties}
For any irreducible representations $\pi$ of $G$, 
the ``growth'' of $\pi$ is recorded in the \emph{associated variety} $\mathrm{As}(\pi)$ and its \emph{annihilator variety} $\mathrm{An}(\pi)$. The algebra $U(\frakg)$ and its ideals are filtered by the degrees, and the \emph{annihilator variety} $\mathrm{An}(M)$ and the \emph{associated variety} $\mathrm{As}(M)$ of a module $M$ are defined as the zero sets of the ideal $\mathrm{Gr}\left(\mathrm{Ann}(M)\right)$ and $\mathrm{Ann}\left(\mathrm{Gr}(M)\right)$, respectively. For groups over $\RR$ and $\CC$, it turns out that the annihilator variety $\mathrm{An}(\pi)$ of any irreducible representation $\pi$ is the closure of an unique orbit $\mathcal{O}(\pi)$ (cf. \cite[3.10]{joseph1985associated}).

\subsection{Induction, Jacquet Functors and Associated Varieties}
In \cite{GourevitchSayag}, Gourevitch and Sayag discussed a family of Lagrangian submanifolds of the annihilator variaties of any irreducible $U(\frakg)$-module. According to \cite[Corollary 6.6]{GourevitchSayag}:
\begin{theorem}
    For any irreducible representation $\pi$ of $G$ and a parabolic subgroup $P=LN$, denoting by $r_P$ the Jacquet functor along the nilpotent radical $N$, then for any irreducible quotient $\tau$ of $r_P(\pi)$,
    \[
        \mathcal{O}(\tau)\subset p_{\frakl}(\mathcal{O}(\pi)\cap\frakn^\perp).
    \]
\end{theorem}
Under the duality given by the Killing form, $\frakg^*$ and $\frakg$ are isomorphic, and the following spaces correspond under such isomorphism:
\begin{center}
    \begin{tabular}{c|c}
        $\frakg$ & $\frakg^*\cong\frakg$\\\hline\hline
        $\frakn^\perp$ & $\frakp$\\\hline
        projection $p_\frakl:\frakp\rightarrow\frakp/\frakn$ & projection $q_\frakl:\frakn^\perp\rightarrow\frakn^\perp/\frakp^\perp$\\\hline
        $\mathcal{O}(\pi)\cap \frakn^\perp$ & $\mathcal{O}(\pi)\cap \frakp$\\\hline
        $\frakn^\perp\cap p_\frakl^{-1}\mathcal{O}(\tau)$ & $\mathcal{O}(\tau)+\frakn$.
    \end{tabular}
\end{center}
Therefore, under the Killing form duality, the following two inclusions are equivalent:
\[
     \mathcal{O}(\tau)\subset p_\frakl \left( \mathcal{O}(\pi)\cap \frakn^\perp\right)\leftrightharpoons \mathcal{O}(\pi)\subset G\cdot( \mathcal{O}(\tau)+\frakn).
\]
\subsection{Examples}
In this section we will give some examples and applications to the Theorem \ref{maintheorem}. 
\subsubsection{$n=3$}
Consider the standard parabolic subgroup $P_{[21]}=L_{[21]}N_{[21]}$ where the Levi subgroup $L_{[21]}=GL(2)\times GL(1)$, and the unipotent radical
\[
    N_{[21]} = \begin{pmatrix}
        1 & 0 & *\\
        0 & 1 & *\\
        0 & 0 & 1
    \end{pmatrix}.
\]
The generalized flag variety $G/P_{[21]}$ is isomorphic to $\mathbb{P}^2$, and the isomorphism can be written explicitly down with the Pl\"{u}cker coordinates
\begin{align*}
    G/P_{[21]}&\longrightarrow \mathbb{P}\left({\wedge^2\mathbb{C}^3}\right)\\
    [g]&\mapsto (g(e_1\wedge e_2),g(e_2\wedge e_3),g(e_3\wedge e_1)).
\end{align*}
\indent For an element $x=\left(\begin{smallmatrix}0&0&0\\0&0&1\\0&0&0 \end{smallmatrix}\right)\in\mathcal{O}_{[21]}$, the generalized Springer fiber $\mathcal{P}_{[21]}(x)$ above $x$ is isomorphic to the projective line $\mathbb{P}^1$, and is given by the equation $X_1=0$ in the Pl\"{u}cker coordinate $[X_0:X_1:X_2]$. The generalized Springer fiber above an element $\left(\begin{smallmatrix}0&1&0\\0&0&1\\0&0&0 \end{smallmatrix}\right)$ is the point $[1:0:0]$.\\
\indent For the case $\alpha = [11]$ and $\beta = [1]$, the two possible choices of $\gamma$ are $[21]$ and $[111]$ with Littlewood-Richardson coefficients $c_{[11],[1]^{[21]}} = c_{[11],[1]^{[111]}} = 1$. The orbit $\mathcal{O}_{[111]}$ is the zero orbit, hence for all parabolic subalgebras $\frakq=\frakm+\fraku$ conjugate to $\frakp_{[21]}$, the variety $p_{\frakm}\left(\mathcal{O}_{[111]}\cap \frakq\right)$ contains the zero orbit $\mathcal{O}_{[11]}\times\mathcal{O}_{[1]}$.\\

\indent The orbit $\mathcal{O}_{[21]}$ can be explicitly described with matrices:
\[
    \mathcal{O}_{[21]} = \left\{A\in\frakg\mid A^2=0\text{ and }A\neq 0\right\}.
\]
We can represent any element in $\mathfrak{p}_{[21]}$ as a matrix of the shape 
$
    \begin{pmatrix}
        A & v\\
        0 & b
    \end{pmatrix}
$
where $v\in\mathbb{C}^2$. The subvariety $\mathcal{O}_{[21]}\cap \mathfrak{p}_{[21]}$ is the following subset
\[
    \mathcal{O}_{[21]}\cap \mathfrak{p}_{[21]} = \left\{\begin{pmatrix}
        A & v\\
        0 & 0
    \end{pmatrix}\neq 0\mid A^2=0, Av=0\right\}
\]
of $\mathfrak{p}_{[21]}$. The elements $[g]$ in the generalized Springer fiber $\mathcal{P}_{[21]}(x)$ are those who satisfy $\mathrm{Ad}(g)^{-1}x\in  \mathcal{O}_{[21]}\cap \mathfrak{p}_{[21]}$. The point $[0:0:1]\in\mathcal{P}_{[21]}(x)$ projects to the zero orbit $\mathcal{O}_{[11]}$ and the affine chart $[1:0:x]$ projects to $\mathcal{O}_{[2]}$.\\

\indent The orbit $\mathcal{O}_{[3]}$ is the set of matrices satisfying
\[
    \mathcal{O}_{[3]} = \left\{A\in\frakg\mid A^3=0,\; A^2\neq 0\text{ and }A\neq 0\right\}.
\]
Therefore, the subvariety $\mathcal{O}_{[3]}\cap\frakp_{[21]}$ is the following set of matrices:
\[
    \mathcal{O}_{[3]}\cap \mathfrak{p}_{[21]} = \left\{\begin{pmatrix}
        A & v\\
        0 & 0
    \end{pmatrix}\neq 0\mid A^3=0, A^2v=0,\; A^2\neq 0\text{ or }Ab\neq 0\right\}.
\]
The generalized Springer fiber over a point of $\mathcal{O}_{[3]}$ is the point $[1:0:0]$, and the whole set $\mathcal{O}_{[3]}\cap \mathfrak{p}_{[21]}$ projects to the orbit $\mathcal{O}_{[2]}$.

\subsubsection{$n=6$}
Consider the Levi blocks $GL(4)\times GL(2) \subset GL(6)$, and consider the orbit $\mathcal{O}_{[22]}\times\mathcal{O}_{[11]}$. Between the induced orbit $[33]$ and $[2211]$, the poset of $G$-orbits contained in $G\cdot (\mathcal{O}_{[22],[11]}+\frakn)$ is displayed as below:
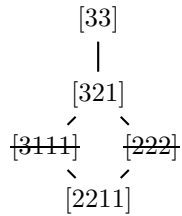
\begin{center}
    \begin{tikzpicture}
        \node (top) at (0,0) {$[33]$};
        \node [below  of=top] (bottom1)  {$[321]$};
        \draw [black,  thick] (top) -- (bottom1);
        \node [below left of=bottom1] (bottomleft2)  {$\stkout{[3111]}$};
        \draw [black,  thick] (bottom1) -- (bottomleft2);
        \node [below right of=bottom1] (bottomright2)  {$\stkout{[222]}$};
        \draw [black,  thick] (bottom1) -- (bottomright2);
        \node [below right of=bottomleft2] (bottom3)  {$[2211]$};
        \draw [black,  thick] (bottomleft2) -- (bottom3);
        \draw [black,  thick] (bottomright2) -- (bottom3);
    \end{tikzpicture}
\end{center}
The orbits $[3111]$ and $[222]$ are excluded from the poset, since $c_{33,2211}^{3111}=c_{33,2211}^{222}=0$.\\
\indent It is important to mention that the $n=6$ case for $P_{[33]}$ is the first case in which we can expect a Littlewood-Richardson coefficient greater than 1. The partial flag variety $G/P_{[33]}$ is isomorphic to the Grassmannian $G(3,6)$ and can be embedded into the projective space $\mathbb{P}(\bigwedge^{3}\mathbb{C}^6)\cong \mathbb{P}^{19}$ with the Pl\"{u}cker embeddings by choosing the basis of $\bigwedge^{3}\mathbb{C}^6$ as $e_i\wedge e_j\wedge e_k$ for each triple $i<j<k$.\\
\indent We will discuss a slightly different space $\mathcal{P}_{\alpha,\beta}(x,y)$ by taking those elements $[g]$ in $\mathcal{P}_{\alpha,\beta}(x)$ such that $p_\mathfrak{l}(\mathrm{Ad}^{-1}(g)x)$ is a fixed element $y\in\mathcal{O}_{\alpha}\times\mathcal{O}_{\beta}$. There is a map
\[
    \sigma: \mathcal{P}_{\alpha,\beta}(x,y)\cdot P\rightarrow \mathrm{Spec}(\CC[\mathfrak{p}]^{L})
\]
which takes every $g$ to $\mathrm{Ad}(g)^{-1}x$, and the ring $\mathbb{C}[\frakp]^L$ is the ring of $\mathrm{Ad}L$-invariant regular functions on $\frakp$ under the adjoint action of $L$. Any element of $\mathrm{Spec}(\CC[\mathfrak{p}]^{L})$ lying in the image of $\sigma$ can be represented by a matrix
\[
    \left(
        \begin{array}{cccccc}
        0 & 1 & 0 & u_{1,4} & u_{1,5} & u_{1,6} \\
        0 & 0 & 0 & u_{2,4} & u_{2,5} & u_{2,6} \\
        0 & 0 & 0 & u_{3,4} & u_{3,5} & u_{3,6} \\
        0 & 0 & 0 & 0 & 1 & 0 \\
        0 & 0 & 0 & 0 & 0 & 0 \\
        0 & 0 & 0 & 0 & 0 & 0 \\
        \end{array}
    \right).
\]
The images of the map $\sigma$ for different choices of $\gamma$ are listed in the following table:
\begin{center}
    \begin{tabular}{ccc|c}
        \hline\hline
        $\gamma$ & $\dim\mathcal{P}_{\alpha,\beta}(x)$ & $c_{[21],[21]}^\gamma$ & Ideal\\\hline
        $[2211]$ & 5 & 1 & $(u_{3,6},u_{2,6},u_{2,4},u_{3,4},u_{1,4}+u_{2,5})$\\\hline
        $[222]$ & 4 & 1 & $(u_{1,4}+u_{3,5},u_{2,4},u_{3,4},u_{2,6})$\\\hline
        $[3111]$ & 4 & 1 & $(u_{3,6},u_{2,6},u_{2,4},u_{3,4})$\\\hline
        $[3,2,1]$ & 2 & 2 & $(u_{3,4},u_{2,4}),(u_{2,6},u_{2,4})$\\\hline
        $[3,3]$ & 1 & 1 & $(u_{2,4})$\\\hline
        $[4,1,1]$ & 1 & 1 & $(u_{2,6}u_{3,4}-u_{2,4}u_{3,6})$\\\hline
        $[4,2]$ & 0 & 1 & $(0)$\\\hline
    \end{tabular}.
\end{center}

\subsubsection{$n=8$}
We look at the Levi blocks $GL(4)\times GL(4) \subset GL(8)$, and consider the orbit $\mathcal{O}_{[22]}\times\mathcal{O}_{[22]}$. Between the induced orbit $[44]$ and $[2222]$, the poset of $G$-orbit $G\cdot (\mathcal{O}_{[22],[22]}+\frakn)$ is displayed as below:
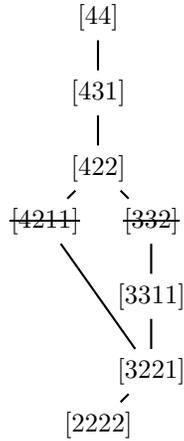
\begin{center}
    \begin{tikzpicture}
        \node (44) at (0,0) {$[44]$};
        \node [below  of=44] (431)  {$[431]$};
        \draw [black,  thick] (44) -- (431);
        \node [below of=431] (422) {$[422]$};
        \draw [black,  thick] (431) -- (422);
        \node [below left of=422] (4211) {$\stkout{[4211]}$};
        \draw [black,  thick] (422) -- (4211);
        \node [below right of=422] (332) {$\stkout{[332]}$};
        \draw [black,  thick] (422) -- (332);
        \node [below of=332] (3311) {$[3311]$};
        \draw [black,  thick] (332) -- (3311);
        \node [below of=3311] (3221) {$[3221]$};
        \node [below left of=3221] (2222) {$[2222]$};
        \draw [black,  thick] (4211) -- (3221);
        \draw [black,  thick] (3311) -- (3221);
        \draw [black,  thick] (3221) -- (2222);
    \end{tikzpicture}
\end{center}
The orbits $[4211]$ and $[332]$ are excluded from the poset, since $c_{22,22}^{4211}=c_{22,22}^{332}=0$.
\subsubsection{$n=12$}
Consider the Levi blocks $GL(6)\times GL(6)\subset GL(12)$ and the orbit $\mathcal{O}_{[321]}\times \mathcal{O}_{[321]}$, if we represent any element of $\mathrm{Spec}(\CC[\frakp]^L)$ by the matrix
\[
    \left(
        \begin{array}{cccccccccccc}
        0 & 1 & 0 & 0 & 0 & 0 & u_{1,7} & u_{1,8} & u_{1,9} & u_{1,10} & u_{1,11} & u_{1,12} \\
        0 & 0 & 1 & 0 & 0 & 0 & u_{2,7} & u_{2,8} & u_{2,9} & u_{2,10} & u_{2,11} & u_{2,12} \\
        0 & 0 & 0 & 0 & 0 & 0 & u_{3,7} & u_{3,8} & u_{3,9} & u_{3,10} & u_{3,11} & u_{3,12} \\
        0 & 0 & 0 & 0 & 1 & 0 & u_{4,7} & u_{4,8} & u_{4,9} & u_{4,10} & u_{4,11} & u_{4,12} \\
        0 & 0 & 0 & 0 & 0 & 0 & u_{5,7} & u_{5,8} & u_{5,9} & u_{5,10} & u_{5,11} & u_{5,12} \\
        0 & 0 & 0 & 0 & 0 & 0 & u_{6,7} & u_{6,8} & u_{6,9} & u_{6,10} & u_{6,11} & u_{6,12} \\
        0 & 0 & 0 & 0 & 0 & 0 & 0 & 1 & 0 & 0 & 0 & 0 \\
        0 & 0 & 0 & 0 & 0 & 0 & 0 & 0 & 1 & 0 & 0 & 0 \\
        0 & 0 & 0 & 0 & 0 & 0 & 0 & 0 & 0 & 0 & 0 & 0 \\
        0 & 0 & 0 & 0 & 0 & 0 & 0 & 0 & 0 & 0 & 1 & 0 \\
        0 & 0 & 0 & 0 & 0 & 0 & 0 & 0 & 0 & 0 & 0 & 0 \\
        0 & 0 & 0 & 0 & 0 & 0 & 0 & 0 & 0 & 0 & 0 & 0 \\
        \end{array}
    \right)
\]
If we set $\gamma=[53211]$, the ideals for the image of the projection map from the 6 dimensional Satake fiber with 4 irreducible components  are the following four ideals:
\begin{align*}
    I_1 &= (u_{6,10},u_{6,7},u_{5,10},u_{5,7},u_{3,10},u_{3,7})\\
    I_2 &=(u_{6,7},u_{5,7},u_{3,7},\\
    &u_{5,12}u_{6,10}-u_{5,10}u_{6,12},u_{3,12}u_{6,10}-u_{3,10}u_{6,12},u_{3,12}u_{5,10}-u_{3,10}u_{5,12}\\
    & u_{3,12}u_{4,7}+u_{3,12}u_{5,8}-u_{2,7}u_{5,12}-u_{3,8}u_{5,12},\\
    & u_{3,10}u_{4,7}+u_{3,10}u_{5,8}-u_{2,7}u_{5,10}-u_{3,8}u_{5,10})\\
    I_3 &= (u_{5,12},u_{5,10},u_{5,7},u_{3,12},u_{3,10},u_{3,7})\\
    I_4 &=(u_{3,12},u_{3,10},u_{3,7},\\
    &u_{5,12}u_{6,10}-u_{5,10}u_{6,12},u_{5,12}u_{6,7}-u_{5,7}u_{6,12},u_{5,10}u_{6,7}-u_{5,7}u_{6,10}\\
    & u_{2,10}u_{6,7}+u_{3,11}u_{6,7}-u_{2,7}u_{6,10}-u_{3,8}u_{6,10},\\
    & u_{2,10}u_{5,7}+u_{3,11}u_{5,7}-u_{2,7}u_{5,10}-u_{3,8}u_{5,10}).
\end{align*}
The Littlewood-Richardson coefficient $c_{[321],[321]}^{[53211]}$ of this case is equal to 4.
\bibliographystyle{alpha}
\bibliography{finalpaper}
\end{document}